\documentclass[12pt]{elsarticle}

\usepackage{amssymb}
\usepackage{amsthm}

\usepackage{xcolor}
\usepackage{soul}

\usepackage[colorlinks,
    linkcolor={red!50!black},
    citecolor={blue!50!black},
    urlcolor={blue!80!black}]{hyperref}

\usepackage{graphicx}
\usepackage{tikz-cd}

\newtheorem{theorem}{Theorem}[section]
\newtheorem{lemma}[theorem]{Lemma}
\newtheorem{proposition}[theorem]{Proposition}

\newtheorem{conjecture}[theorem]{Conjecture}

\newtheorem{definition}[theorem]{Definition}

\newcommand{\TryPackage}[3]{\IfFileExists{#1.sty}{\usepackage{#1}#2}{#3}}
\TryPackage{mathrsfs}{\renewcommand{\mathcal}{\mathscr}}{%
     \TryPackage{eucal}{}{}}

\renewcommand{\rho}{\varrho}

\newcommand{\CC}{{\mathbb C}}

\newcommand{\SLC}{{SL(2, {\mathbb C})}}

\begin{document}
\begin{frontmatter}

\title{Weights of essential surfaces in 2-bridge knot complements}



\author[1]{Cynthia L. Curtis\corref{Corresponding author}}
\ead{ccurtis@tcnj.edu}

\author{Kendra Ebke\fnref{Kendra}}
\ead{ebke@wisc.edu}

\author{Kate O'Connor\fnref{Kate}}

\ead{koconnor1127@gmail.com}

\affiliation{organization={Department of Mathematics and Statistics}, addressline={The College of New Jersey}, city={Ewing}, postcode={08628}, state={NJ}, country={USA}}
\cortext[Corresponding author]{Corresponding author}
\fntext[Kendra]{Current Address: Mathematics, University of Wisconsin, Madison, WI}
\fntext[Kate]{Current Address:Ignite Fueling Innovation, Inc, 310 The Bridge St, Suite 450, Huntsville, AL 35806}






\begin{abstract}
Understanding ideal points in the character varieties of knot complements has led to a number of important invariants for 3-manifolds. Ohtsuki (1994) counted the ideal points for character varieties of 2-bridge knot complements, and he made his techniques more concrete in an ensuing paper (1996). Drawing on these ideas, for all 2-bridge knots $K$, we explicitly determine the structure of a Serre tree for each essential surface in the knot complement directly from the knot diagram. Using these trees, we derive a formula for the number of ideal points associated to each incompressible surface.


\end{abstract}


\end{frontmatter}


\section{Introduction}
Culler and Shalen's seminal work \cite{CS} established the central role played by essential surfaces in the understanding of 3-manifolds. In particular, a number of 3-manifold invariants related to character varieties, which have proved crucial to our understanding of 3-manifolds in recent decades, rely heavily on these essential surfaces, including the seminorms of Culler, Gordon, Luecke, and Shalen \cite{CGLS} (henceforth referred to as CGLS-seminorms); A-polynomials \cite{CCGLS}; and $\SLC$-Casson invariants \cite{C01}, \cite{C03}. The computation of these invariants remains challenging. 

We briefly describe the relationships at the heart of the work of Culler and Shalen. Recall that the character variety of the fundamental group of a 3-manifold is the space of characters of irreducible representations of the fundamental group in $SL_2(\mathbb{C})$. This space is seen to be a variety in a natural way. Moreover, if the character variety is 1-dimensional, then the variety 
is a punctured Riemann surface. The puncture points - i.e. the points of the Riemann surface which do not correspond to characters of representations - are known as \textit{ideal points} of the character variety. Culler and Shalen show that an ideal point of the character variety corresponds to an action of the fundamental group on a certain abstract tree. Further, this action gives rise to a splitting of the fundamental group, and this in turn gives rise to an incompressible surface with boundary in the 3-manifold. This surface is not unique, but its boundary curve is well defined up to isotopy. 

Given the complexity of these constructions and the abstract nature of the trees in question, computations relying on this theory can be difficult. In \cite{O1}, Ohtsuki works through this theory for 3-manifolds which are the complements of 2-bridge knots. In his work, he begins a process of explicitly understanding the geometry of the trees in question, enabling explicit computations. He carried this process further in his later paper \cite{O2}.

The purpose of our paper is to build upon \cite{O2}, showing how concrete trees may be constructed directly from a knot diagram. Each such tree corresponds to a particular incompressible surface in the knot complement, and the number of associated ideal points can be computed directly from the tree
. We hope this process could eventually lead to combinatorial computations for any Montesinos knot. 

In this paper we carry out this project in full for 2-bridge knots, reproducing the main computational result of \cite{O1}. In so doing, we show how to find trees associated to each twist in the standard 2-bridge knot diagram, and we provide an understanding of how these trees combine to form the various trees associated to incompressible surfaces in the knot complement. In addition we show how to count the ideal points associated to a particular surface 
using the tree constructed. 

We note that as Montesinos knots are similarly assembled from twists, this process should be generalizable. To illustrate this,
we conclude the paper by constructing trees which we conjecture are the correct trees corresponding to incompressible surfaces in the complements of all $(p,q,r)$ pretzel knots, where $p$, $q$, and $r$ are odd and greater than 1.

The paper is organized as follows. In Section 2, we review known results concerning essential surfaces in the complements of 2-bridge knots. In Section 3, we offer a brief introduction to Serre trees. In section 4, we describe the trees associated with essential surfaces in 2-bridge knot complements, showing how to generate the various trees for each essential surface directly from a knot diagram. We also see how to count the ideal points corresponding to each essential surface. In Section 5 we show that our work reproduces the result of \cite{O1}. Finally in Section 6, we use this process to generate conjectural trees for an infinite family of pretzel knots.

We note that the research for Sections 4 and 5 was completed by the first and third authors in support of the undergraduate thesis of the third author. All three authors contributed to the development of the trees in Section 6.

\section{Essential surfaces, trees, and weights} 
In this section we review known results on the essential surfaces in two-bridge knot complements, enumerate these in a convenient way, and reframe Ohtsuki's main computational result \cite{O1}, Proposition 5.2.

\subsection{Essential surfaces in two-bridge knot complements}

  A surface $\Sigma$ in a 3-manifold $M$ is said to be \textit{incompressible} if for any disk $D \subset M$ with $D \cap \Sigma = \partial D$, there exists a disk  $D' \subset \Sigma$, with $\partial D' = \partial D$.  A surface $\Sigma$ is $\partial${\em - incompressible} if for each disk $D \subset M$ with $D \cap \Sigma = \partial_+D$ and $D \cap \partial M  = \partial_-D$ there is a disk $D' \subset \Sigma$ with $\partial_+ D' =\partial_+ D$ and $\partial_- D'  \subset \partial \Sigma$. A surface $\Sigma \subset M$ is \textit{essential} if it is both incompressible and $\partial$-incompressible. 
 
 From \cite{HT} we have a good understanding of the essential surfaces in 2-bridge knot complements. Let $K = K(\alpha , \beta)$ be a 2-bridge knot, and let
 
 $$\frac{\beta}{\alpha} = r_0 + \frac{1}{b_1 + \frac{1}{b_2 + \ldots +\frac{1}{b_j}}}$$
 be a continued fraction expansion of $\beta/\alpha$.  Henceforth we denote such a continued fraction expansion by $[b] = [b_1,b_2, \ldots, b_j].$ 
 
Hatcher and Thurston show that any incompressible surface is supported by a branched surface corresponding to a continued fraction expansion of $\beta/\alpha$. Moreover they show that the continued fractions which yield incompressible surfaces are precisely those with $|b_i|\geq 2$ for every $i$.
 
 In \cite{CFLM}, Proposition 3.3, the authors show that the continued fractions corresponding to branched surfaces which support incompressible surfaces may equivalently be enumerated as follows. 
Let $[n]=[n_1,n_2, \ldots, n_k]$ 
 be the unique continued fraction expansion of $\frac{\beta}{\alpha}$ with all $n_i$ 
 positive.  Set   $\epsilon_0=\epsilon_{k+1}=0$. 
\begin{samepage} Let $S$ be the set of $k$-tuples $(\epsilon_1, \epsilon_2, \ldots, \epsilon_k)$ with $\epsilon_i = 0 \mbox{ or } 1$ satisfying both of the following conditions:
\begin{enumerate}
    \item If $\epsilon_i = 1$, then $\epsilon_{i+1}=0$ and $\epsilon_{i-1} = 0$.
    \item If $n_i =1$ and $\epsilon_i=0$ then either $\epsilon_{i-1} = 1$ or $\epsilon_{i+1}=1$.
\end{enumerate} \end{samepage}
 We define
\begin{definition}
 Let $s = (\epsilon_1, \epsilon_2, \ldots, \epsilon_k) \in S$. The \textbf{continued fraction expansion of $\beta/\alpha$ generated by }$s$ is the continued fraction expansion obtained by the following steps:
 \begin{enumerate}
     \item If $\epsilon_j = 1$ and $j>1$, replace $n_j$ with an alternating sequence $-2,2,-2,\ldots,\pm2$ of length $n_j-1$. If $n_j=1$, this means that the term is simply deleted.
     \item If $\epsilon_j=1$, add 1 to each of $n_{j-1}$ and $n_{j+1}$. If $j=1$ then only $n_2$ is adjusted, and if $j=k$ then only $n_{k-1}$ is adjusted.
     \item Adjust the sign(s) of all term(s) replacing $n_i$ in the resulting sequence by multiplying by $\prod_{j<i, \epsilon_j=1} (-1)^{n_j}.$
 \end{enumerate}
 \end{definition}
 In this way, the elements of $S$ generate precisely those continued fraction expansions corresponding to the Hatcher-Thurston branched surfaces supporting incompressible surfaces.

 We now reinterpret this geometrically, using the notion of state surfaces. Each boundary slope of a given 2-bridge knot is the boundary slope of exactly one so-called state surface. State surfaces are constructed as follows: 
 
 Consider the standard 2-bridge diagram for the knot $K$ corresponding to the all-positive continued fraction expansion   $[n_1,n_2, \ldots, n_k].$ Note this is an alternating diagram, as all signs are positive.  We smooth each crossing by replacing the crossing with either horizontal or vertical arcs, as illustrated in Figure~\ref{fig:smoothings}. We thereby obtain circles (``state circles''), which we fill in with disks. These disks are then connected with half-twisted bands corresponding to the crossings of the knot.

\begin{figure}
\begin{center}
\leavevmode\hbox{}
\includegraphics[scale=0.2]{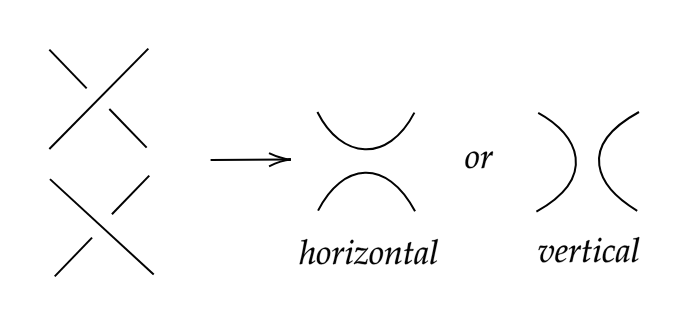}
\caption{Horizontal and Vertical Smoothings} 
\label{fig:smoothings}
\end{center}
\end{figure}

This process builds a surface whose boundary is the 2-bridge knot, called the \textit{state surface} associated with the ``state'' given by a choice of smoothings for each crossing of the given 2-bridge knot. Note that this process yields a finite number of state surfaces associated with a given 2-bridge knot. It is immediately clear from this construction and from that of \cite{HT} that each Hatcher-Thurston branched surface corresponds to a state surface.

An example of a state surface associated with the $[5,4,3,6]$ 2-bridge knot is shown in Figure~\ref{fig:surfaceexample}. Here the largest disc with boundary the outer circle should be understood to lie at a level below the discs nested within that circle.

\begin{figure}
\centering
\includegraphics[width=9cm, height=4cm]{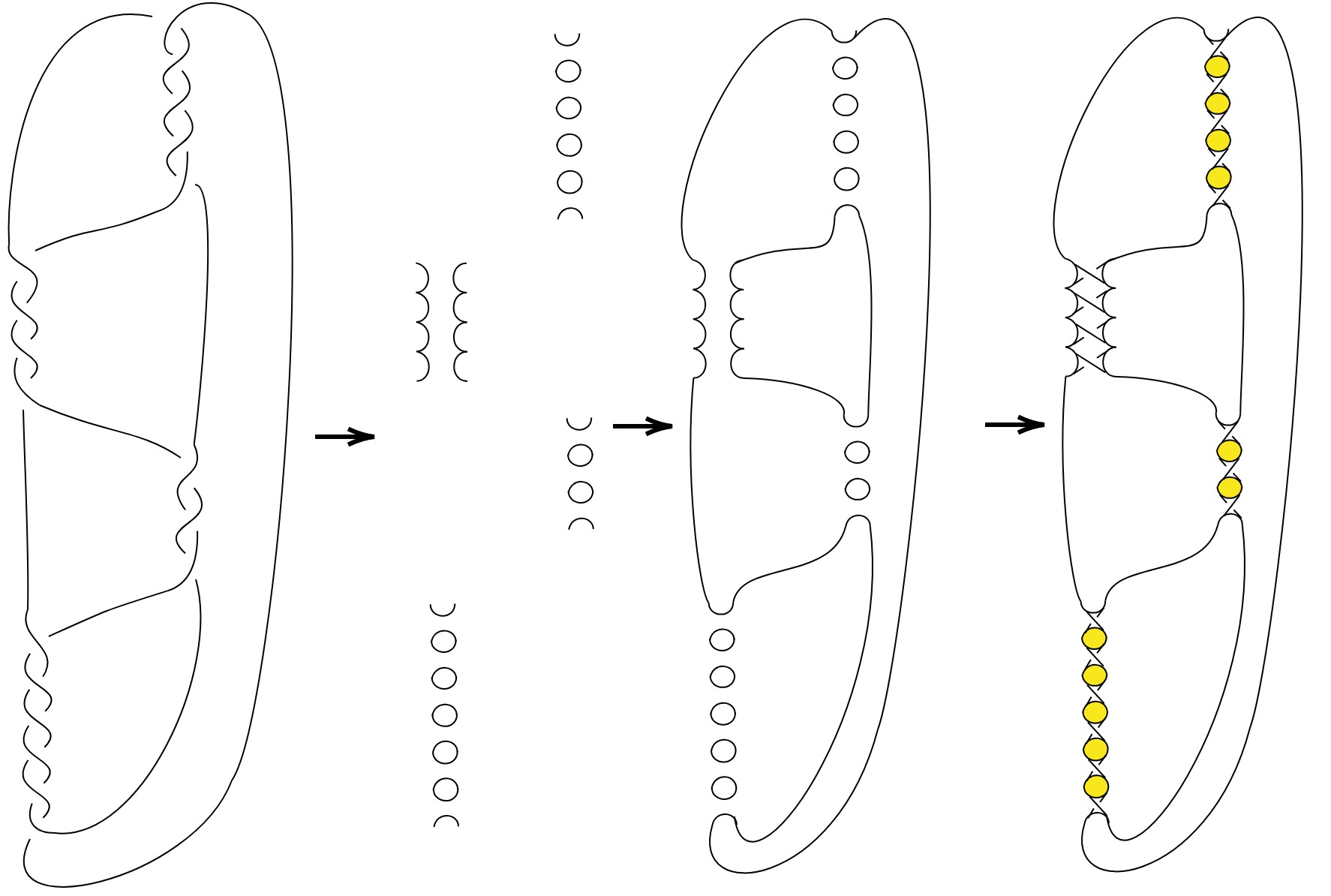}
\caption{A state surface of the $[5,4,3,6]$ knot, where all of the crossings in the first, second, and fourth twists are smoothed horizontally, and all of the crossings in the third twist are smoothed vertically.}
\label{fig:surfaceexample}
\end{figure}

Now the incompressibility of state surfaces has been studied by Ozawa in \cite{Oz}. In particular, if a crossing of the knot joins a state circle to itself in the above process, then the surface constructed is compressible. It is immediately clear from this that any state leading to an incompressible state surface must have all crossings within each twist smoothed in the same way. Thus, to identify all incompressible state surfaces, we need only choose a smoothing for each twist. If we let 0 denote the horizontal smoothing and 1 the vertical smoothing, we see the choices of smoothings are $k$-tuples of 0's and 1's. The elements of $S$ are precisely the $k$-tuples yielding incompressible surfaces. Restating the conditions on $S$ geometrically, we obtain 
\begin{lemma}
A state surface of a 2-bridge knot is essential if and only if the following conditions on the smoothings hold.

\begin{enumerate}
    \item All of the crossings in an individual twist are smoothed the same way (either horizontally or vertically).
    \item Two adjacent twists are not both smoothed vertically.
    \item If a twist has a single crossing, either that crossing or an adjacent twist is smoothed vertically.
\end{enumerate}
\end{lemma}

\begin{definition} We call the elements of $S$ the \textbf{allowable smoothings} of the alternating diagram corresponding to the all-positive continued fraction expansion of $\frac{\alpha}{\beta}$.
\end{definition}

Note that the state surface of the $[5,4,3,6]$ knot given above corresponds to the smoothing (0,1,0,0), which is an allowable smoothing. Thus the surface is an essential surface of the knot.

\subsection{Ohtsuki's Theorem}
We next turn our attention to the CGLS-seminorm. The CGLS-seminorm is a seminorm on the vector space $V$ spanned 
by  $\mathbb{Z}\times \mathbb{Z} = H_1(\partial M; \mathbb{Z})$, where $M$ is a 3-manifold with boundary a torus. 

Intuitively, if the character variety of $M$ is 1-dimensional, this seminorm gives a count of the  characters in the (0-dimensional) character variety of the closed 3-manifold obtained by $p/q$-Dehn surgery. More precisely, let $\textit{X}$ be the character variety of $M$. If $m$ and $\ell$ are a meridian and longitude, we define $I_{p,q}:X \rightarrow \mathbb{C}$ to be the function taking the character of a representation $\rho$  to the trace of $\rho(m^p\ell^q)$. The CGLS-seminorm on $V$ takes $(p,q)$ to the degree of $I_{p,q}$.

In \cite{O1}, Proposition 5.2, Ohstuki computes the seminorm for $M$, where $M$ is the complement of a 2-bridge knot. Let $K(\alpha,\beta)$ be a 2-bridge knot. Let $[m] = [m_1,m_2,...m_j]$ be a continued fraction expansion of $\beta/\alpha$ corresponding to an incompressible surface. Let $N_{[m]}$ be the boundary slope of the associated incompressible surface. Note since $K(\alpha,\beta)$ is a 2-bridge knot, $N_{[m]}$ is an integer. Ohtuski tells us that the
$||(p,q)|| = \sum_{[m]} |p-N_{[m]} q| \cdot \big[ \frac{1}{2} \prod_i (|m_i|-1) \big],$ 
where the sum is taken over all continued fraction expansions $[m]$ corresponding to incompressible surfaces (\cite{O1}, Proposition 5.2). However, as noted in \cite{Ma}, p. 18 and \cite{BC12}, p. 12, this formula contains a minor error. Ohtsuki himself notes in section 4 of his paper that the correct coefficient of $|p|$ is $-\frac{1}{2} +\frac{1}{2} \prod_i (|m_i| - 1)$. Thus, the correct formula for the seminorm is 
$$||(p,q)|| = \frac{1}{2}(-|p| + \sum_{[m]} |p-N_{[m]} q| \prod_i (|m_i|-1) ).$$

Observe that $|p-N_{[m]} q|$ is the geometric intersection of the boundary of the incompressible surface with the curve $pm + q\ell$. Thus the CGLS-seminorm of $(p,q)$ is a weighted sum of these intersection numbers, where the sum is taken over all incompressible surfaces in the knot complement. This fact holds true for CGLS-seminorms of knot complements broadly, not only for 2-bridge knot complements. 

Note that the set of boundary slopes for a 2-bridge knot, and in fact for any Montesinos knot, is well understood and may be computed combinatorially. Thus to understand the CGLS-seminorm and to produce combinatorial computations for the seminorms, we must better understand the weights associated to each surface.
Equivalently, we wish to easily count the number of ideal points associated to each incompressible surface in the knot complement. Ohtsuki has accomplished this for 2-bridge knots.

We reframe Ohtuski's result in terms of the set $S$ of allowable smoothings.

\begin{proposition}

\label{OhtsukiSmoothings}
The number of ideal points of the character variety associated with the surface given by smoothing $(\epsilon_1, \epsilon_2, \ldots, \epsilon_k)$ is

\[\frac{1}{2}\left(\gamma + \prod_{i=1}^{k} \delta_i\right)\]

where $\delta_i$ is given by

 \[ \delta_i  = \left\{ \begin{array}{ll}
      n_i - 1 + \epsilon_{i-1} + \epsilon_{i+1} & \mbox{ if  $\epsilon_i = 0$}\\
      1 & \mbox{ if $\epsilon_i = 1$},
   \end{array} \right.
\]

and $\gamma$ is given by

\[ \gamma =  \left\{ \begin{array}{rl}
      -1 & \mbox{ if the surface is orientable}\\
      0 & \mbox{ if the surface is non-orientable}
   \end{array} \right.
\]

\end{proposition}

\begin{proof}
    The smoothing $(\epsilon_1, \epsilon_2, \ldots, \epsilon_k)$ generates a continued fraction expansion $[m]$ from the all-positive continued fraction expansion $[n]=[n_1,n_2,\ldots, n_k]$ according to the  replacement steps given in Definition 2.1. 
    By Ohtsuki's formula, the weight for the corresponding surface is $ \frac{1}{2} \prod (|m_i|-1)$ if the boundary slope is not 0, and the weight of the surface with boundary slope 0 is $ \frac{1}{2}(-1 + \prod (|m_i|-1))$.

    Now note that as we generate $[m]$ from $[n]$,  if $\epsilon_i = 1$ then $n_i$ is replaced by a sequence of terms $\pm 2$. These terms each contribute a factor of 1 to the product in Ohtsuki's formula. We see too that $\delta_i = 1$ for such $i$.
    
    Now if $\epsilon_1 = 0$ then $n_i$ is replaced by a single term of $[m]$. Let $d_i$ denote the term in $[m]$ corresponding to $n_i$, where $\epsilon_i = 0$. The product in Ohtuski's formula is then $\prod(|d_i|-1).$
    Checking the replacement steps in Definition 2.1, we see $d_i$ is equal to $n_i$ if both $\epsilon_{i-1}$ and $\epsilon_{i+1}$ are 0, is equal to $n_i+1$ if exactly one of $\epsilon_{i-1}$ and $\epsilon_{i+1}$ is 1, and is equal to $n_i+2$ if both $\epsilon_{i-1}$ and $\epsilon_{i+1}$ are 1. In other words, the contribution of $d_i$ to the product in Ohtsuki's formula is $n_i + \epsilon_{i-1} + \epsilon_{i+1} - 1$. Then in each case $\delta_i = |d_i|-1.$
    
    Finally noting that $|p-N_{[m]}q| = |p|$ precisely when the boundary slope $N_{[m]}$ is 0, we see the correction to Ohtsuki's formula applies when the surface in question is a Seifert surface for the knot. This occurs when the surface is orientable, and $\gamma$ provides the needed correction to the coefficient of $|p|$ in this case.

    \end{proof}

\begin{definition}
We call the number of ideal points associated with the surface with smoothing $(\epsilon_1, \epsilon_2, \ldots, \epsilon_k)$ the \textbf{weight} of the surface.
    
\end{definition}

Thus, by Proposition~\ref{OhtsukiSmoothings} the weight of a surface corresponding to $(\epsilon_1, \epsilon_2, \ldots, \epsilon_k)$ is given by
\[\frac{1}{2}\left(\gamma + \prod_{i=1}^{k} \delta_i\right).\]

\section{Serre trees}

In this section we briefly summarize the results needed from \cite{CS} and \cite{O2}. In \cite{CS}, Section 2, Culler and Shalen introduce three key theorems. Together, Theorems 2.1.2, 2.2.1, and 2.3.1 tell us the following: given an ideal point, the fundamental group of the 3-manifold is isomorphic to the fundamental group of a graph of groups arising from an action of the fundamental group on a certain tree, known as a Serre tree. This isomorphism gives a splitting of the fundamental group, and in fact there is a family of incompressible surfaces within the three manifold such that the fundamental group of each surface is contained in an edge group.

To make this a bit more concrete: this tells us that each ideal point in the character variety gives rise to an action of the fundamental group on a tree. Moreover there is an associated  family of incompressible surfaces in the 3-manifold, and the fundamental group of the 3-manifold acts on the tree in such a way that the surface group stabilizes an edge of the tree.  

We refer to the tree on which the group acts corresponding to a given ideal point as the Serre tree for the ideal point. Formal definitions for Serre trees may be found in section 2 of \cite{CS}. The definition is abstract, and we will identify Serre trees concretely in the remainder of the paper without reference to this abstract definition.

In \cite{O1} and \cite{O2}, Ohtsuki works to identify these trees concretely. It is this approach which we will develop here.

Ohtsuki gives an alternative definition of Serre's tree in Definition 3.1 of \cite{O2} using formal power series. One sees that each generator of the knot group is associated with an axis in the tree such that multiplication by the generator acts on the axis  by translation. In particular the axis is invariant under multiplication by the generator. Following \cite{O2}, for the remainder of the paper, we describe the trees in question by drawing only the axes associated to a set of generators. These axes completely determine the Serre tree as described in \cite{O2}.

Also apparent from this definition is that there is a special vertex in the tree, which we call the origin, which lies on all axes, and which will be useful for us later on. In fact it is part of the edge of the tree which is stabilized by the action of the surface group corresponding to the ideal point.

Ohtuski proves the following lemma concerning these axes (Lemma 1.5 of \cite{O2}). This will be our primary tool throughout the next section.  We have rephrased this slightly.

\begin{lemma}
\label{minimalsubtrees}
Given a relation $z = xyx^{-1}$ in the knot group, the  subtrees shown in Figure~\ref{fig:minimalsubtrees} containing the six ends $x_{\pm}, y_{\pm}, z_{\pm}$ of the axes associated with $x$, $y$, and $z$ must be contained in any tree associated with the knot. 

\begin{figure}
\centering

\includegraphics[width=7cm, height=2cm]{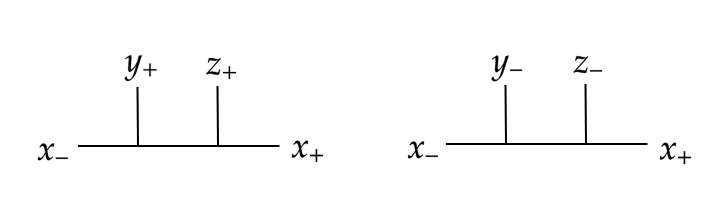}
\caption{The subtrees obtained from the relation $z = xyx^{-1}$.}
\label{fig:minimalsubtrees}
\end{figure}
\end{lemma}

\begin{definition} Given a set of generators and relations for the knot group, we call the trees given by Lemma~\ref{minimalsubtrees} \textbf{minimal subtrees} for the presentation. 
    
\end{definition}

In \cite{O2}, the author's approach is to use his lemma to ``guess'' the structure of the Serre trees corresponding to ideal points for a given knot. He then shows these are in fact Serre trees by identifying families of representations leading to an ideal point which realizes the tree.

We take a slightly different approach. We use Ohtsuki's lemma to find the structure of all candidate trees which could arise as Serre trees corresponding to ideal points of the character variety for each 2-bridge knot. We do this by determining the configuration of axes for each tree directly from the 2-bridge knot diagram. 

For each candidate tree, we then identify a special segment which will be fixed by the surface group of the incompressible surface corresponding to the ideal point. We determine which surface group fixes the segment with the given action on the axes, thereby identifying the incompressible surface corresponding to the tree. We remark that for two-bridge knots, every candidate tree arising from the minimal subtrees for the knot is a Serre tree.

In summary, in what follows, for each incompressible surface in the knot complement, we find the axes corresponding to a set of generators for all trees on which the knot group acts, with the caveat that the fundamental group of the surface lies in an edge stabilizer. In fact, for each surface we will find a maximal such tree, with other trees and actions arising from this maximal one. Then the weight of the surface is the number of such actions on this maximal tree.

\section{Trees and weights}

We now turn our attention to our main goal of building a tree for each incompressible surface in a 2-bridge knot complement directly from a knot diagram. We then use these trees to find the weights for each surface.

Given a 2-bridge knot $K$, consider the diagram for $K$ corresponding to the all-positive continued fraction expansion $[n]$ for $K$. We describe the fundamental group of the knot complement using the Wirtinger presentation for this knot diagram.

We proceed as follows: In Section 4.0.1, we find all possible subtrees associated to a twist in a knot diagram, and in Section 4.0.2, we determine which tree is associated with which smoothing of the twist. In Section 4.1 we determine the structure of the tree associated to the state surface given by the all-horizontal smoothing of the alternating diagram given by the all-positive continued fraction expansion (assuming this smoothing is allowable), and we determine the weight for the surface in 4.1.2. In Section 4.2 we complete the same work for the remaining allowable smoothings of the alternating diagram.

\subsubsection{Subtrees corresponding to twists in the knot diagram}

To determine the shape of the trees associated to a particular incompressible surface, we must find all ways to fit together the minimal subtrees obtained from each of the relations in the knot group. To simplify this task, we first fit the minimal subtrees obtained from the crossings of an individual twist, obtaining subtrees corresponding to the twist. 

Consider the open twist illustrated in Figure~\ref{fig:opentwist}.
\begin{figure}
\centering

\includegraphics[width=2cm, height=4cm]{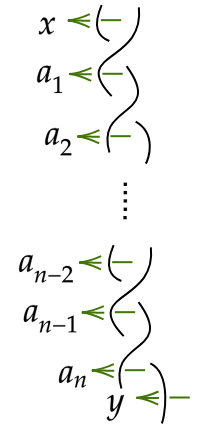}
\caption{A labeled open twist.}
\label{fig:opentwist}
\end{figure}
From the open twist in Figure~\ref{fig:opentwist}, we obtain the $2n$ minimal subtrees shown in Figure~\ref{fig:opentwistminimalsubtreees}. 

\begin{figure}
\centering

\includegraphics[width=10cm, height=3cm]{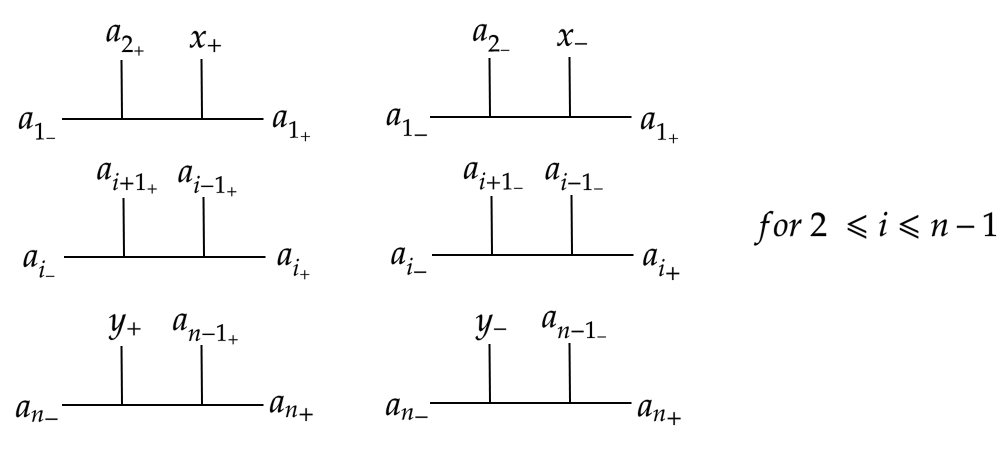}
\caption{Minimal subtrees obtained from an open twist.}
\label{fig:opentwistminimalsubtreees}
\end{figure}

Now consider the two trees shown in Figure~\ref{fig:opentwistsubtreees}. It is straighforward to check that each of the minimal subtrees from Figure~\ref{fig:opentwistminimalsubtreees} is contained in each of these composite trees. 

Looking at the parasol tree shown,  we may generate additional composite trees by identifying two or more branches in the parasol. This will create a new parasol with fewer branches and with more axes along each branch. Note that the resulting tree might or might not contain all of the minimial subtrees from Figure~\ref{fig:opentwistsubtreees}, as one or more of the subtrees could be collapsed in the glued parasol.

\begin{definition}
  A \textbf{quotient} of a tree $T$ is a tree  $T'$ obtained from $T$ by an identification of some branches in $T$. A composite tree arising from a set of subtrees is \textbf{maximal} for a set of minimal subtrees if it contains all of the minimal subtrees and it is not a proper quotient of another tree containing the given subtrees.
\end{definition}

\begin{proposition}
\label{Subtree}
There are exactly two maximal ways to fit together the minimal subtrees obtained from the open twist in Figure~\ref{fig:opentwist}. These are given in Figure~\ref{fig:opentwistsubtreees}.
\begin{figure}
\centering

\includegraphics[width=15cm, height=5cm]{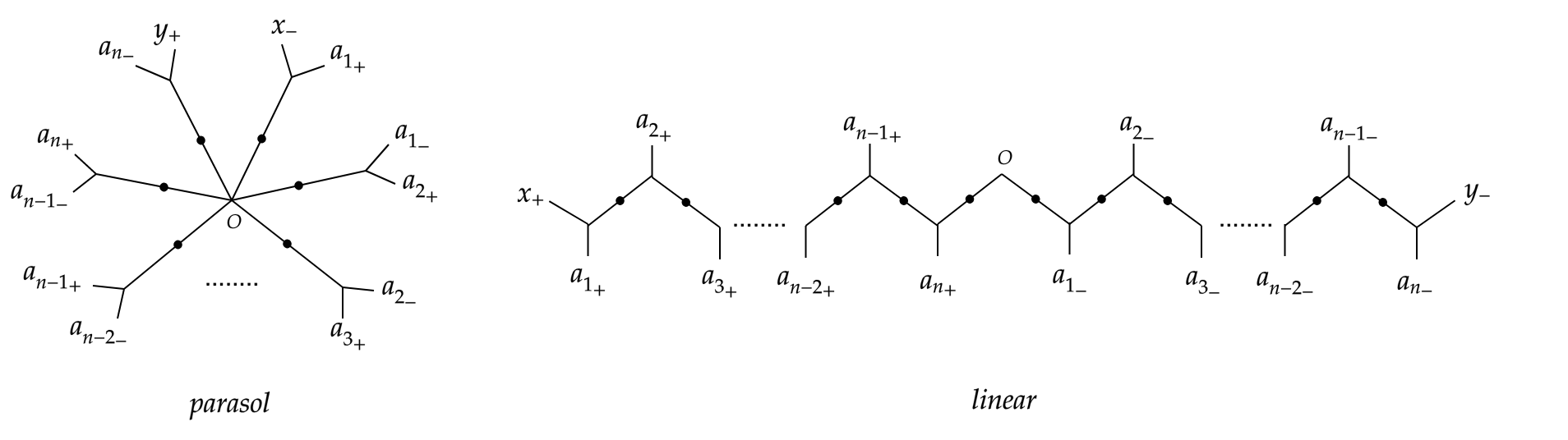}
\caption{Maximal subtrees associated with an open twist.}
\label{fig:opentwistsubtreees}
\end{figure}

\end{proposition}

\begin{proof}
As noted, all of the minimal subtrees are contained in the parasol and linear subtrees in Figure~\ref{fig:opentwistsubtreees}. Furthermore, any tree which contains all of the minimal subtrees is a quotient of one of the subtrees in Figure~\ref{fig:opentwistsubtreees}, which can be shown by exhaustive guessing and checking. This process of guessing and checking is accelerated when we recall that the origin must lie on each axis of the tree. Choosing a position of the origin in any one of the minimal subtrees quickly forces the remaining choices, and one arrives at either a parasol or a linear composite tree structure.
\end{proof}

As we move on, we will need to think more carefully about the group action on each of these trees. Note that in \cite{O2} the vertices in the minimal subtrees are understood to be a single unit apart. For our purposes, to avoid issues when passing to the double cover of the surface group, we will view the action given by multiplication by a generator as a movement of two units along the tree rather than one unit. Accordingly we view the vertices shown in the minimal subtrees as lying two units apart. We do not draw the extra vertex in the minimal subtrees, but these extra vertices are shown as we begin to draw composite trees.

Also note that the minimal subtrees in Figure~\ref{fig:opentwistminimalsubtreees} do not completely determine the placement of $y_-$ and $x_+$ in the parasol subtree and $x_-$ and $y_+$ in the linear subtree, so these ends are omitted from the subtrees given in Figure~\ref{fig:opentwistsubtreees}.

\subsubsection{Subtrees corresponding to smoothings of twists}
Next we consider how choices of smoothings affect the structure of the subtree associated with the twist.

\begin{proposition}
\label{twistsubtrees}
If a twist is smoothed vertically, then the associated maximal subtree is a linear subtree. If a twist is smoothed horizontally, the associated maximal subtree is a parasol subtree. There are three possible structures for the parasol subtree corresponding to that twist.

\end{proposition}

\begin{proof}
We first assume that the twist is smoothed vertically. 

Recall that when we smooth a twist vertically, the portion of the essential surface corresponding to that twist is as shown in Figure~\ref{fig:vertloop}. We see the words of the form $a_j^{-1}a_i$ are in the fundamental group of the surface, and so these must all stabilize an edge in the tree.

\begin{figure}
\centering

\includegraphics[width=4cm, height=4cm]{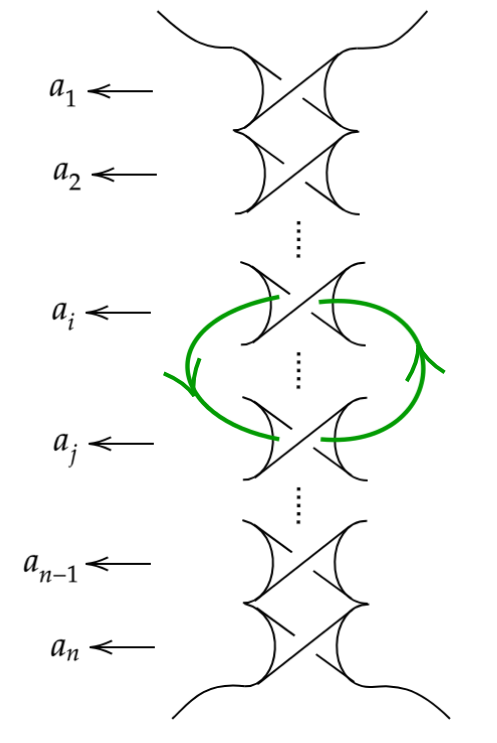}
\caption{The word corresponding to the loop in green is $a_j^{-1}a_i$.}
\label{fig:vertloop}
\end{figure}

We check that all words of the form $a_j^{-1}a_i$ fix the segment of the tree joining the origin $O$ to the first vertex along the $a_{1-}$-branch of the linear tree.

We begin at the origin, reading the word from right to left. The action given by multiplication by a generator is to translate two units along the corresponding axis. So we first move two units toward $a_{i_+}$ on the $a_i$ axis, then move toward $a_{j_-}$ on the $a_j$ axis. This brings us back to $O$, so the origin is fixed by the word $a_j^{-1}a_i$. Similarly, the vertex one unit from $0$ along the $a_{1-}$-branch is fixed by any such word, and hence the entire segment is fixed by any word of the form $a_j^{-1}a_i$.

In contrast, note that the origin is not fixed by words of the form $a_j^{-1}a_i$ in the parasol subtree. Therefore if a twist is smoothed vertically, then the corresponding subtree must be a linear subtree.

Now assume that the twist is smoothed horizontally. One example of such a smoothing for a two bridge knot is shown in Figure~\ref{fig:hhhcurve}. Here we show a horizontal smoothing of the twist in which the adjacent twists are also smoothed horizontally. 

\begin{figure}
\centering

\includegraphics[width=4cm, height=4.5cm]{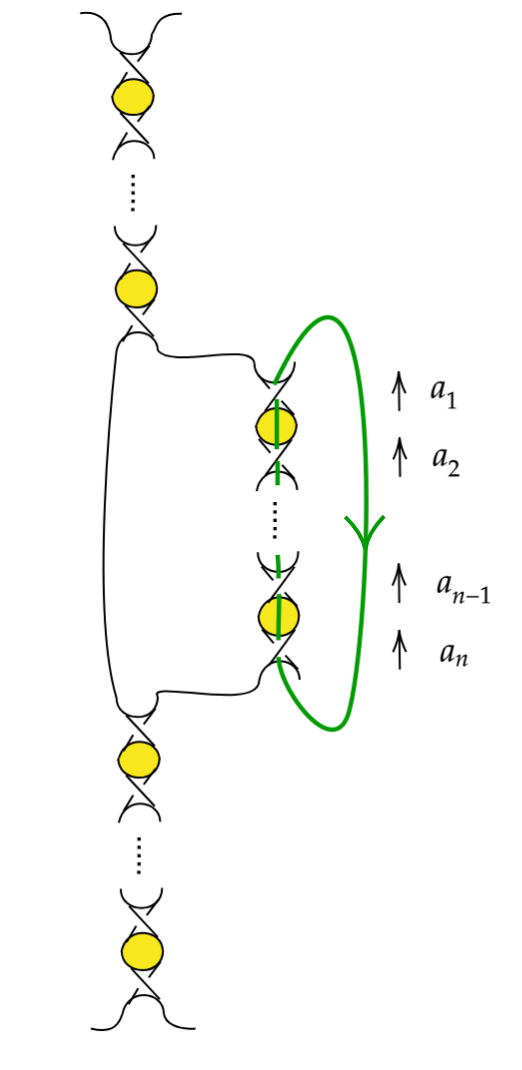}
\caption{The green curve is the simplest loop that goes through all of the crossings of the $k^{th}$ twist.}
\label{fig:hhhcurve}
\end{figure}

It is clear that any loop on the surface passing along this twisted band will follow the word $a_1a_2...a_{n-1}a_n$. In  Figure~\ref{fig:hhhcurve}, the ends of the green curve entering and exiting the twisted band shown lie in the same disk, and the green curve shown lies in the surface. However this will not always be the case, and in general a closed loop on the surface which passes along the twisted band will correspond to a word or the form $wa_1a_2...a_{n-1}a_n$ for some $w$. For 2-bridge knots, there are cases corresponding to four possible words $w$, with parasol structures shown in Figure~\ref{fig:allsubtreecases}.

\begin{figure}
\centering

\includegraphics[width=7cm, height=7cm]{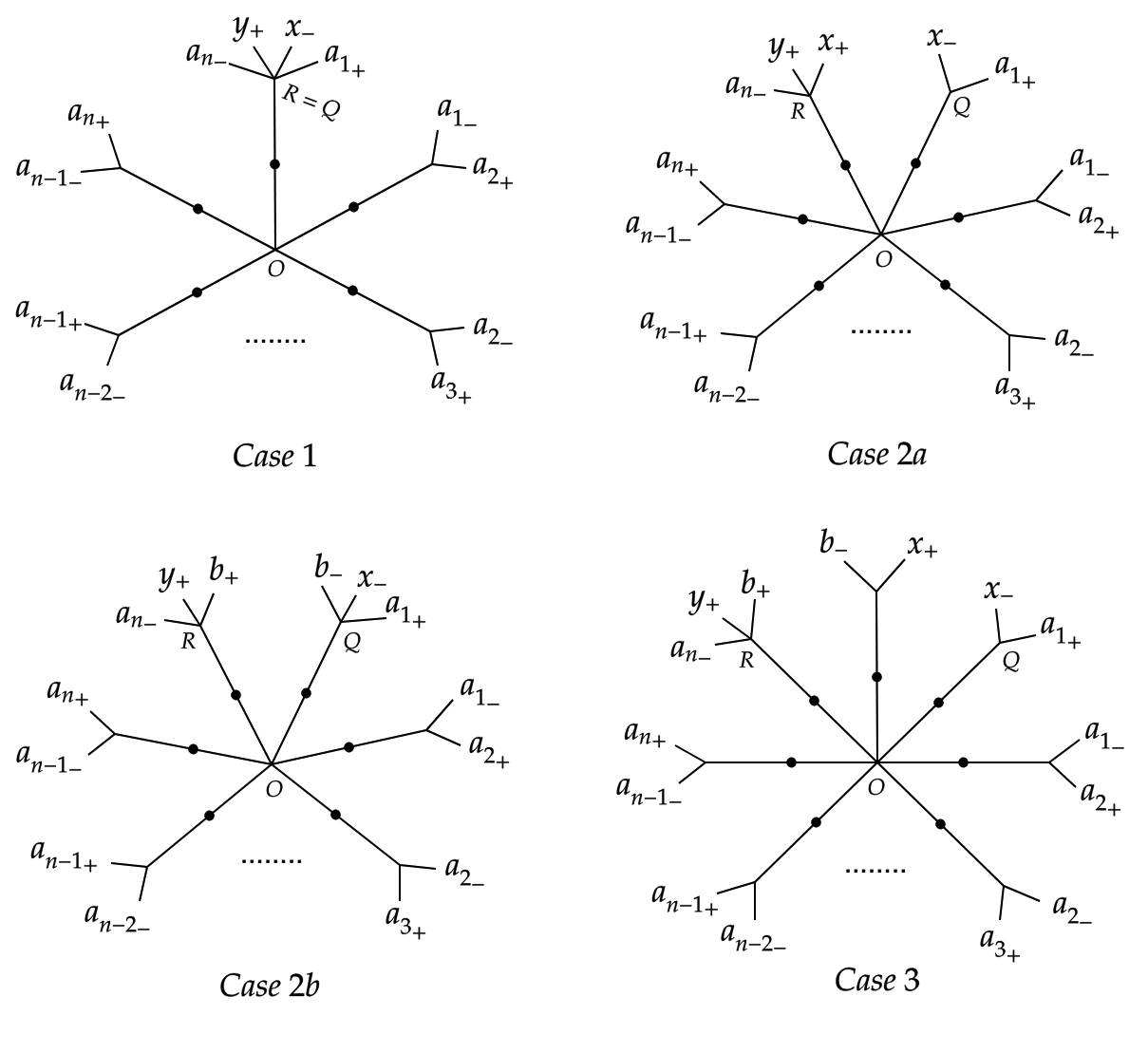}
\caption{Possible parasol subtree structures.}
\label{fig:allsubtreecases}
\end{figure}

These cases arise according to the way in which adjacent twists are smoothed. Specifically, Case 1 occurs when both adjacent twists are smoothed horizontally. In this case the word $w$ is trivial.

Case 2a occurs when the twist immediately before is smoothed vertically and the twist immediately after is smoothed horizontally. In this case $w = x$, where $x$ is the generator moving through the bottom crossing of the twist above (oriented right to left).  Case 2b occurs when the twist immediately before is smoothed horizontally and the twist immediately after is smoothed vertically. Here $w=b$, where $b$ is the generator moving from left to right through the top crossing of the twist immediately after the $a$-twist.

Case 3 occurs when both adjacent twists are smoothed vertically. Here $w = bx$.

When the first twist of a knot is smoothed horizontally, the parasol subtree has the structure of Case 1 if the second twist is smoothed horizontally and the structure of Case 2b if the second twist is smoothed vertically. Similarly, when the last twist of a knot is smoothed horizontally, the parasol has the structure given in Case 1 or Case 2a, depending only on the smoothing of the second-to-last twist.

If the wordlength of the word $wa_1a_2...a_{n-1}a_n$ is even, then either the surface is orientable or the loops described are loops in the orientable double cover. The reader may check that both $O$ and $R$ (and hence the segment joining them) in each parasol subtree is fixed by the corresponding word for that case. 

If the wordlength of $wa_1a_2...a_{n-1}a_n$ is odd, then the loop described is not closed in the oriented double cover of the surface, and instead $wa_1a_2...a_{n-1}a_nwa_1a_2...a_{n-1}a_n$ corresponds to a simplest closed loop through the twisted band. Again the reader may confirm that this word fixes the segment from $O$ to $R$ using the appropriate parasol subtree. 

Finally note that the words do not fix the origin in the linear subtree. It follows that when the twist is smoothed horizontally, the corresponding subtree is a parasol.

\end{proof}

\subsection{All Horizontal Smoothing} \label{horizontalsection}

For each twist and for each allowable smoothing, we now know the maximal subtrees corresponding to the twist.  It remains to assemble these maximal subtrees to build a tree associated to the surface. To do this, recall that by Lemma \ref{minimalsubtrees}, for each crossing of the knot diagram there must be two minimal subtrees included in the tree. Most of these have already been incorporated into the subtrees corresponding to the twists. 

\begin{definition} A \textbf{gluing minimal subtree} is a minimal subtree involving generators from multiple twists in the knot diagram.
\end{definition}

These gluing minimal subtrees provide the remaining information needed for assembling our parasol and linear subtrees corresponding to the twists into a tree for the surface. 

Recall that the tree we draw corresponding to each surface will show only the axes corresponding to our chosen generators. It is the maximal such tree upon which the knot group can act with the requirement that the surface group acts as an edge stabilizer. 
\begin{definition}
We refer to this tree as the \textbf{basic tree} for the incompressible surface.
    
\end{definition}

We first construct the basic tree corresponding to the all-horizontal smoothing of a 2-bridge knot. Let $K$ be a 2-bridge knot with all-positive continued fraction expansion $[n_1,n_2,\ldots,n_k]$. We assume that $|n_i|\geq 2$ for all $i$, so that the all-horizontal smoothing is allowable.  Our tree will be assembled from the $k$ subtrees corresponding to the twists in the associated knot diagram. 

For the essential surface obtained by smoothing all of the twists horizontally, each of these subtrees is a parasol subtree. Call the origins of these subtrees $O_1, O_2, \ldots, O_k$, where $O_1$ corresponds to the origin of the subtree of the first twist, $O_2$ corresponds to the origin of the subtree of the second twist, and so on. Define sets $A$ and $B$ such that $A = \{O_i \mid i \mbox{ odd}\}$ and $B = \{O_j \mid j \mbox{ even}\}$. In the basic tree associated with the all horizontal smoothing of $K$, all of the origins in set $A$ are placed at a single vertex called $P_A$, and all of the origins in set $B$ are placed at a second vertex called $P_B$. Join $P_A$ and $P_B$ by a segment of length two units, and call the vertex between them $O$, the origin of the basic tree. Place the parasols for the twists in odd positions so that the vertex of the parasol is at $P_A$ and so that the vertex $R$ from  Case 1, Figure \ref{fig:allsubtreecases} is at $P_B$. Place all parasols for the twists in even positions so that the vertex of the parasol is at $P_B$ and the vertex $R$ from Case 1, Figure \ref{fig:allsubtreecases} is at $P_A$. The remaining branches of all of the subtrees are distinct. 

To illustrate, in Figure~\ref{fig:5346hhhh}  we see the fully assembled tree corresponding to the all-horizontal smoothing of the $[5,4,3,6]$ 2-bridge knot. This is assembled using the gluing minimal subtrees shown in Figure \ref{5346gluingminimal}. The reader may confirm that the parasols have the structure given in Case 1 of Figure~\ref{fig:allsubtreecases} and that the gluing minimal subtrees appear correctly in the tree. 

Now note that the fundamental group of the surface is generated by the words found in Section 4.0.2 for each twist. The segment from $P_A$ to $P_B$ is fixed by all such words. Thus, the surface group stabilizes this edge.

\begin{figure}
\centering
\includegraphics[width=6cm, height=8cm]{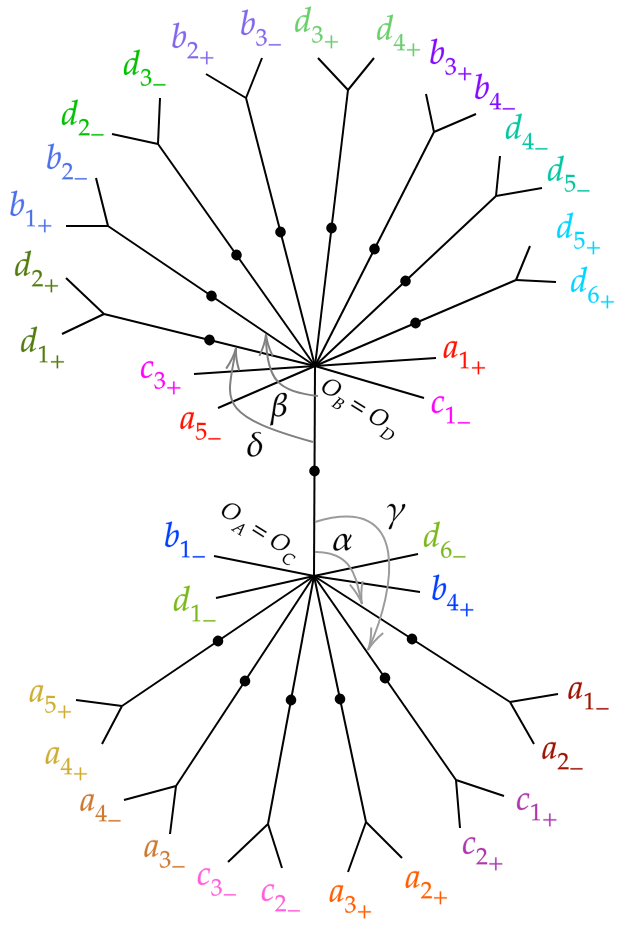}
\caption{Basic tree corresponding to the all horizontal smoothing of the $[5,4,3,6]$ knot.}
\label{fig:5346hhhh}

\end{figure}

\begin{figure}
\centering
\includegraphics[width=12cm, height=6cm]{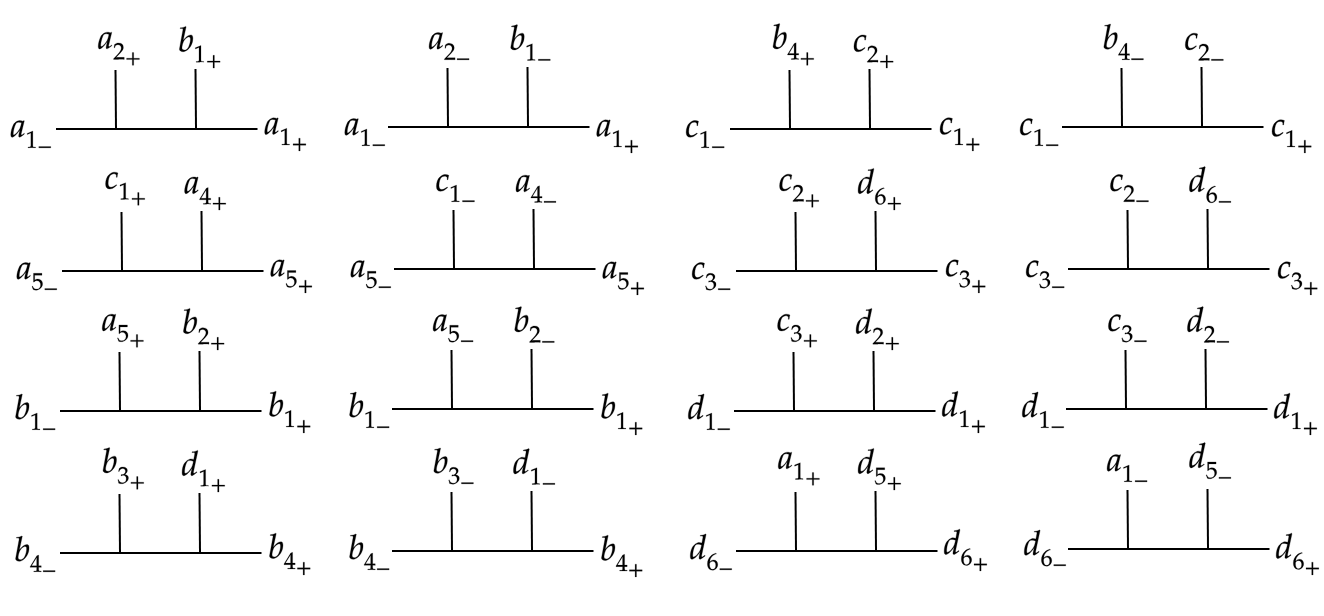}
\caption{Gluing minimal subtrees corresponding to $[5,4,3,6]$ knot.}
\label{5346gluingminimal}

\end{figure}

To summarize so far: we have found the basic tree associated to the surface arising from smoothing all crossings horizontally for any 2-bridge knot. The labels on the axes describe the corresponding group action.

\subsubsection{Additional actions on the basic tree}

It remains to describe the additional group actions associated to the surface. Any other group action associated to the surface must still obey all that we have established here. In fact, the full family of group actions associated to the surface is easily deduced by allowing the group to act a bit differently on the parasols in the basic tree, as we now describe. 

Consider the parasols shown in Figure~\ref{fig:allsubtreecases}. Since all smoothings are horizontal in this case, so that all parasols fall under Case 1, 
the valence of the origin for each parasol subtree of the basic tree is the number of crossings of the twist, $n_i$. We may think of each parasol as lying in a plane, letting the angle between adjacent branches at the origin be $\frac{2\pi}{n_i}$.   In Figure~\ref{fig:5346hhhh} we see the angles for the first, second, third, and fourth parasols, respectively, are $\alpha = \frac{2\pi}{5}$, $\beta = \frac{2\pi}{4}$,  $\gamma = \frac{2\pi}{3}$, and $\delta = \frac{2\pi}{6}$. More generally, for the surface corresponding to the all horiziontal smoothing of the knot $[n_1,n_2, ...,n_k]$, we have $k$ such angles, $\alpha_1,\alpha_2,\ldots,\alpha_k$.
 
 Now to obtain new group actions on $T$, we can multiply $\alpha_i$ by $1, 2, \ldots, j_i -1$. Visually, this changes how the $i^{th}$ parasol subtree sits in $T$. We may obtain a tree isopmorphic to $T$ with adjusted labels. Alternatively, the new tree may be a proper quotient of $T$ if we multiply $\alpha_i$ by an integer $m$ such that $m$ and $n_i$ are not relatively prime and $gcd(m, n_i) \not = \frac{n_i}{2}$. In this case the new parasol will have $\frac{j_i}{gcd(m, j_i)}$ branches at $O_i$. The branches of the new parasol will include portions of the axes for more generators than before.
 
 
 We can do this independently for each of the $i$ parasol subtrees of $T$ with $1 \leq i \leq k$. To illustrate this idea, in Figure~\ref{fig:5346hhhhquotient} we show the quotient tree where $\alpha, \beta,$ and $\gamma$ are unchanged and $\delta$ is multiplied by 2. Note that there are 6 branches at the origin of the subtree associated with the D twist, and $gcd(6,2) = 2$. So the tree corresponding to an action where $\delta$ is multiplied by 2 will be a quotient of Figure~\ref{fig:5346hhhh} with only three branches in the new parasol associated to the $D$-twist. 

\begin{figure}
\centering
\includegraphics[width=6cm, height=7cm]{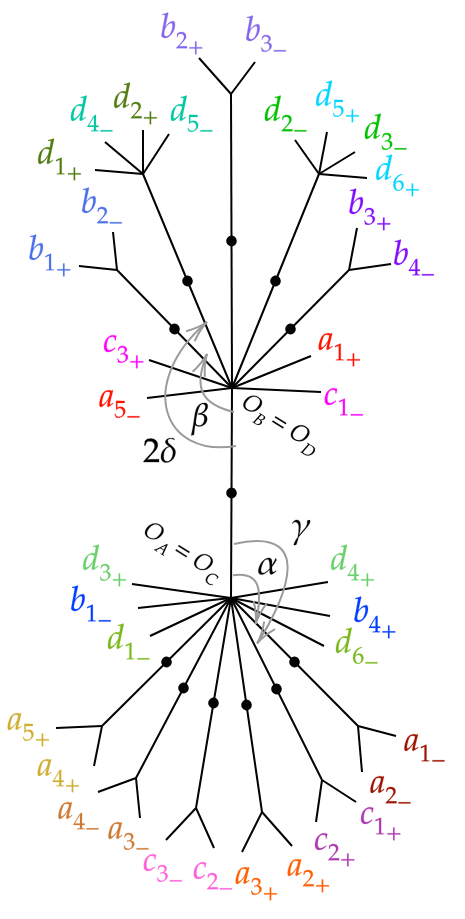}
\caption{Multiplying $\delta$ by 2.}
\label{fig:5346hhhhquotient}

\end{figure}

We must make a slight adjustment if we multiply $\alpha_i$ by an integer $m$ where $gcd(m,n_i)=\frac{n_i}{2}.$ (Note that this implies $j_i$ is even). This changes the angle between adjacent branches in the resulting $i^{th}$ subtree to $\pi$. The resulting subtree is no longer a precise quotient of the basic tree. Instead, in this case we move the origin $O_i$ of the subtree  to the vertex $O$ between $P_A$ and $P_B$. All of the axes in the subtree then contain the segment from $P_A$ to $P_B$, and the axes diverge from $P_A$ and $P_B$ as they do from the origin in the original parasol subtree.  An example of this for the tree associated with the all-horizontal smoothing of the $[5,4,3,6]$ knot, where $\beta = \frac{2\pi}{4}$ is multiplied by 2, is shown in Figure~\ref{fig:5346hhhhbroom}.

\begin{figure}
\centering
\includegraphics[width=6cm, height=7cm]{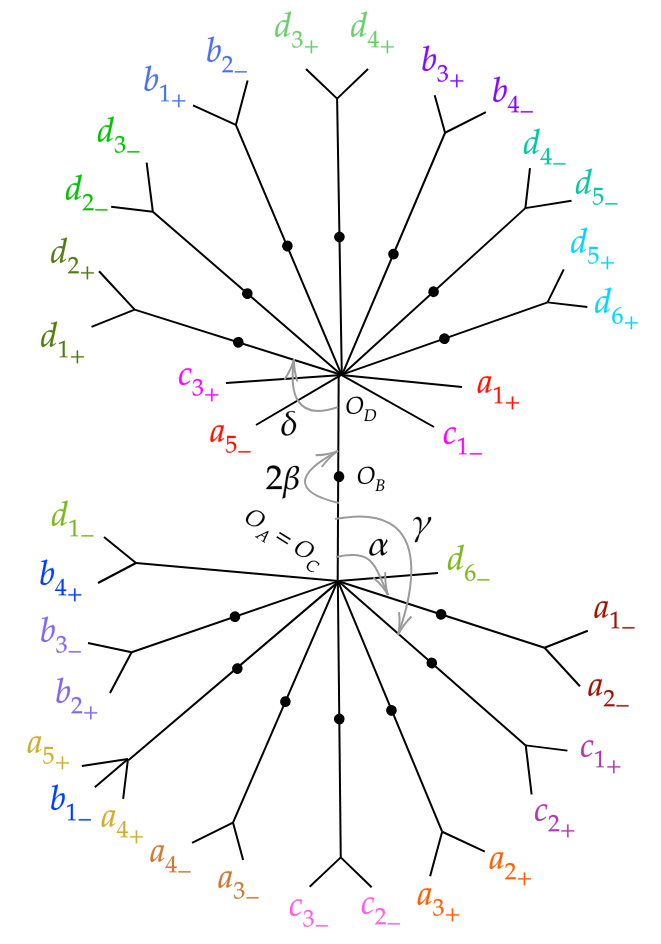}
\caption{Multiplying $\beta$ by 2.}
\label{fig:5346hhhhbroom}

\end{figure}

This process generates all possible trees corresponding to irreducible group actions of the fundamental group of the knot such that the surface group for the surface arising from smoothing each crossing horizontally stabilizes an edge containing the origin. We use this fact in Section 5 to find the precise weight of the surface.

\subsection{Other smoothings} \label{othersmoothingsection}

\subsubsection{Basic Trees}

In the general case, we obtain the essential surfaces of a 2-bridge knot by smoothing each twist either horizontally or vertically, according to the choice of allowable smoothing from the set $S$.  Let $S_i$ denote the subtree corresponding to the $i^{th}$ twist of the knot, with corresponding origin $O_i$.  As above, let $A = \{O_i \mid i \mbox{ odd}\}$ and $B = \{O_j \mid j \mbox{ even}\}$.  If the $i^{th}$ twist of the knot is smoothed vertically, move $O_i$ to the opposite set, forming new sets $A'$ and $B'$ , or possibly a single set $A'$. If two sets are obtained, form a segment of length two units, with endpoints labeled $P_{A'}$ and $P_{B'}$. If a single set is formed, let $ P = P_{A'} = P_{B'}$ be a point.

We first consider the subtrees corresponding to each twist. Each twist smoothed horizontally will contribute a parasol subtree of one of the types shown in Figure~\ref{fig:allsubtreecases}, with the precise nature of the parasol determined as in the proof of Proposition \ref{twistsubtrees}. Each twist smoothed vertically will contribute a linear tree. 

Let $E$ be a branch of a linear tree lying on one side of the origin.  Note that the branch includes a branch of an axis corresponding to a generator from another twist. This generator also appears on a branch of the parasol or linear tree corresponding to this second twist. This branch will overlap with $E$ in the assembled tree.

With this in mind, assemble the tree as follows. Place each linear tree  so that the origin $O_i$ lies at $P_{A'}$ or $P_{B'}$, according to whether the origin is in the set $A'$ or $B'$ (or at $P$ if $P_{A'} = P_{B'}$.)  One end of the linear tree will contain a generator from a tree with an origin in the opposite set. Place this end of the linear tree so that it includes the segment $P_{A'}P_{B'}$ as its first two segments. If two linear trees have ends which overlap initially, their origins will lie in the same set, and the branches will divide two units along from this origin.

Now place parasols for the remaining twists with origins at $P_{A'}$ or $P_{B'}$ (or $P$), as appropriate. If a branch of a parasol overlaps with an end of a linear tree, and if the branch of the parasol contains a generator which does not yet appear on the end of the linear tree, place the generator on the end of the linear tree, two units along the end from the origin of the parasol. 

To illustrate this, we show the maximal trees for the (0,1,0,0) and the (1,0,1,0) smoothings of the [5,4,3,6] knot if Figures \ref{fig:5346hvhh} and \ref{fig:5346vhvh}, respectively. 

Note in the first tree, one end of the $b$-linear tree coincides with the $d_{1+},d_{2+}$ branch of the $d$-parasol and the other end of the $b$-linear tree coincides with the branch of the Case 2a parasol for the $a$-twist which includes $R$. Note here that $d_{2+}$ and $c_{1-}$ have been placed two units away from the points $O_D$ and $O_C$, respectively. 

In the second tree we see that $P_{A'}=P_{B'}$, that the $a$-tree overlaps with the $c$-tree on one end and with the branch of the Case 3 $b$-parasol containing the point $Q$ on the other end. The $c$-linear tree overlaps with branches from the $b$- and $d$-parasols.

One checks that the resulting tree contains all gluing minimal subtrees.

\begin{figure}
\centering
\includegraphics[width=7cm, height=7cm]{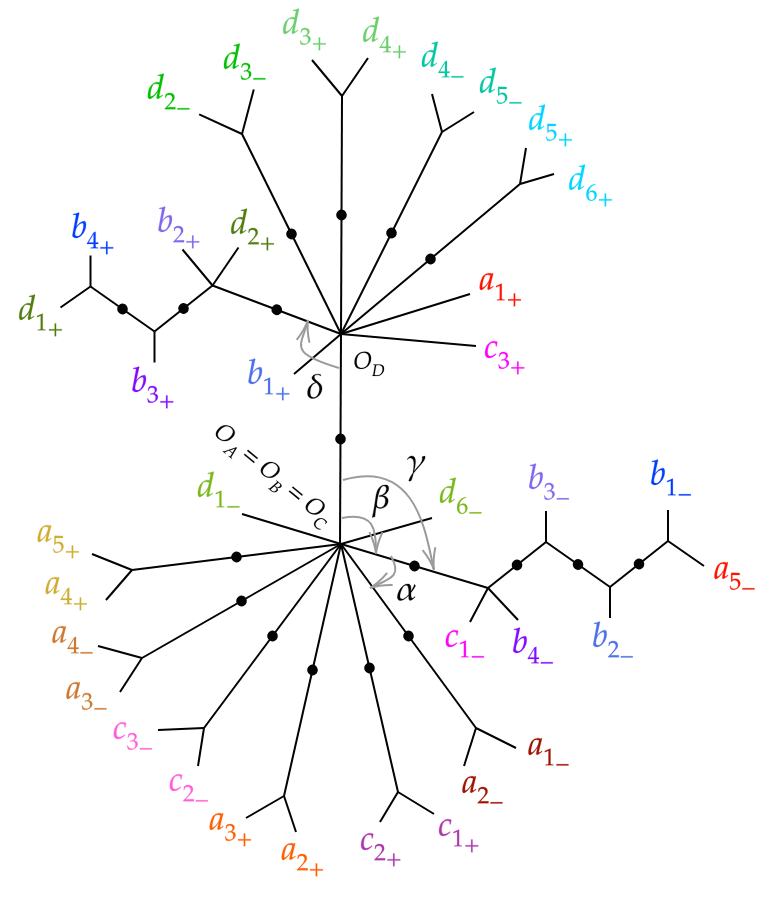}
\caption{Basic tree corresponding to the horizontal-vertical-horizontal-horizontal smoothing of the $[5,4,3,6]$ knot.}
\label{fig:5346hvhh}

\end{figure}

\begin{figure}
\centering
\includegraphics[width=9cm, height=7cm]{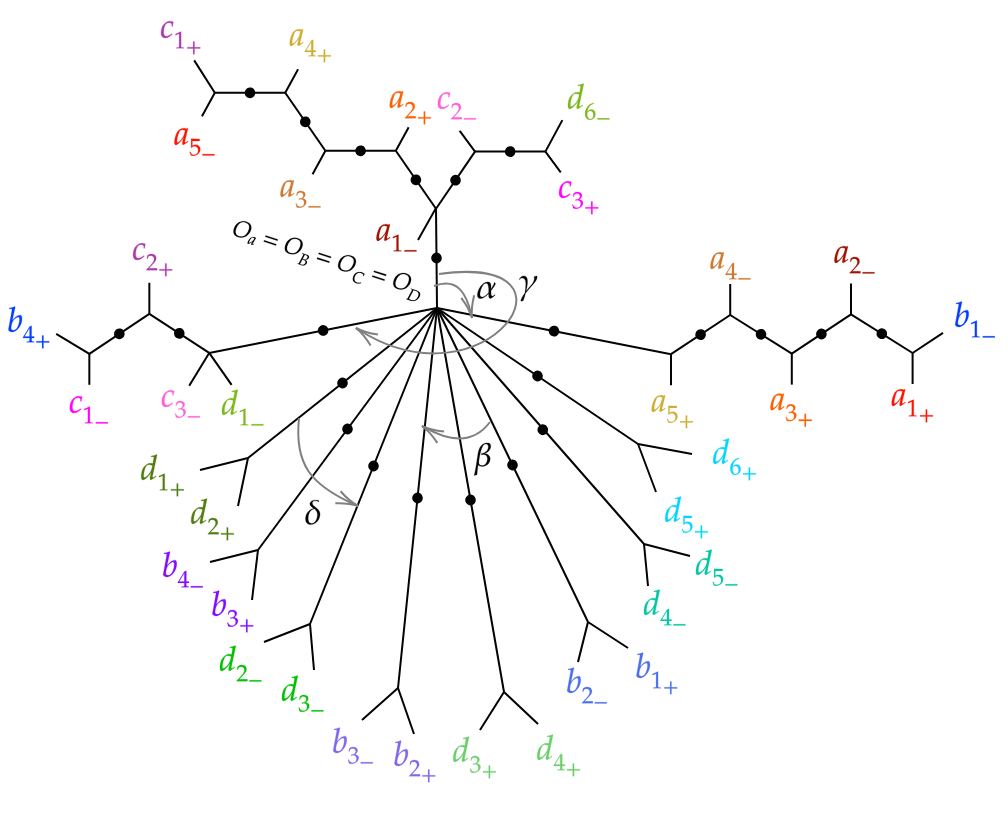}
\caption{Basic tree corresponding to the vertical-horizontal-vertical-horizontal smoothing of the $[5,4,3,6]$ knot.}
\label{fig:5346vhvh}

\end{figure}

\subsubsection{Additional group actions}

As in the case of the all horizontal smoothing, we see that any group action for which the surface group acts trivially must act on the basic tree corresponding to the surface found in 3.3.1. And as in the earlier case, we see the only possible variations are to multiply the angles for the various subtrees by any number from 1 to one less than the valence of the subtree. For linear subtrees, there is then no choice of action, as the valence of the origin is 2. For parasol subtrees note that the valence of the origin of the parasol subtree in Case 1 is the number of crossings of the twist, $n_i$. The valence of the origin in Cases 2a and b is $n_i+1$, and here we let the angle between adjacent branches be $\frac{2\pi}{n_i+1}$. The valence of the origin in Case 3 is $n_i+2$, and we set the angle between adjacent branches to be $\frac{2\pi}{n_i+2}$.  

It is straightforward to obtain most of these group actions, as in the case of the all horizontal smoothing. The trees obtained are either equivalent to the basic tree or are quotients of the basic tree in which the only aspect of the tree that changes is the identification of some branches of parasol subtrees. 

As in the all-horizontal case, we must make one additional adjustment in order to multiply the angle of a subtree with an even number $p$ of branches by $m = \frac{p}{2}$, so that the new tree obtained has an angle of $\pi.$  Call the origin of this subtree $O_i$. If this parasol corresponds to a twist in which an adjacent twist is smoothed vertically, then $O_i$ moves one unit along the linear subtree corresponding to the adjacent twist smoothed vertically. The direction in which $O_i$ moves is determined by the gluing minimal subtrees. 

\begin{proposition}
For any smoothing, there exists a tree realizing each choice of angles given by the basic tree. 
\end{proposition}

Thus the various possible group actions associated to the surface arise from all possible choices of angles for the various parasol subtrees within the basic tree for the surface.

\section{Weights of incompressible surfaces in 2-bridge knot complements}

To complete our analysis of 2-bridge knots, we explicitly count the group actions associated to each surface. In the previous section we saw that the various group actions correspond to choices of angles for each of the subtrees associated to twists. However there is some duplication. We show:

\begin{theorem}
\label{weightsfromvalences}
Let $K$ be a 2-bridge knot with continued fraction expansion $[n_1, n_2, \ldots, n_k]$. Let $s$ be an allowable smoothing for $K$, with associated incompressible state surface $\Sigma_s$, and let $T$ be the basic tree for this surface. The weight of $\Sigma_s$ is $$\frac{1}{2} (\gamma +\prod_i (v_i -1)),$$ where $v_i$ is the valence of the subtree of $T$ corresponding to the $i^{th}$-twist and 

\[ \gamma =  \left\{ \begin{array}{rl}
      -1 & \mbox{ if the surface is orientable}\\
      0 & \mbox{ if the surface is non-orientable.}
   \end{array}\right.
\]

\end{theorem}

\begin{proof}
We show the weight given in Theorem~\ref{weightsfromvalences}   agrees with that given in Theorem~\ref{OhtsukiSmoothings}. Let $s = (\epsilon_1, \epsilon_2,...,\epsilon_k)$. Let $T_i$ be the subtree of $T$ corresponding to the $i^{th}$-twist for $1 \leq i \leq k$.

For each subtree $T_i$,  we have 

\[v_i = \left\{ \begin{array}{ll}
2 & \mbox{if $T_i$ is linear}\\
n_i & \mbox{if $T_i$ is a Case 1 parasol}\\
n_i + 1 & \mbox{if $T_i$ is a Case 2 parasol}\\
n_i + 2 & \mbox{if $T_i$ is a Case 3 parasol.}
\end{array}\right.
\]

We know $T_i$ is a parasol if $\epsilon_i=0$ and is linear if $\epsilon_i = 1$. Further if $\epsilon_i=0$, the parasol $T_i$ is of type 

$$ \begin{array}{ll}

\mbox{Case 1} & \mbox{if $i=1$ and $\epsilon_{i+1}=0$}\\
& \mbox{or $1 < i < k$ and $\epsilon_{i-1} = 0$ and $\epsilon_{i+1}=0$}\\
& \mbox{or $1 = k$ and $\epsilon_{i+1}=0;$}\\
\mbox{Case 2} & \mbox{if $i=1$ and $\epsilon_{i+1}=1$}\\
& \mbox{or $1 <i < k$ and exactly one of $\epsilon_{i\pm 1} = 1$}\\
& \mbox{or $1 = k$ and $\epsilon_{i+1}=1;$}\\
\mbox{Case 3} & \mbox{if $1 < i <k$ and $\epsilon_{i-1} = 1$ and $\epsilon_{i+1} = 1$.}
\end{array} $$

Since the number of possible angles for each tree $T_i$ is one less than $v_i$, we find we have exactly 


\[ v_i -1 =  \left\{ \begin{array}{ll}
      n_i - 1 + \epsilon_{i-1} + \epsilon_{i+1} & \mbox{ if  $\epsilon_i = 0$}\\
      1 & \mbox{ if $\epsilon_i = 1$.}
   \end{array} \right.
\]
choices of angle for $T_i$. Thus  $v_i-1 = \delta_i$, where $\delta_i$ is tas defined in  Thoerem~\ref{OhtsukiSmoothings}.

Now, recall that the character variety is the variety of irreducible representations of the fundamental group, and therefore we count only those actions which are irreducible. A group action is \textit{reducible} if there is a subtree which is invariant under the group action.

In our case, the only reducible group action in the count above arises when every angle in the tree is chosen to be $\pi.$ In this case the group action on the quotient tree fixes a line in the tree. This is possible only if the $n_i$'s are all even, which means that the surface is orientable. Thus, if the surface is orientable, we remove one angle combination. This gives us the formula

\[\gamma + \prod_{i=1}^{k} \delta_i\]

where

\[ \gamma =  \left\{ \begin{array}{rl}
      -1 & \mbox{ if the surface is orientable}\\
      0 & \mbox{ if the surface is non-orientable}
   \end{array}\right.
\]

Finally, note that if the trees realizing two group actions are reflections of one another, then the group actions are conjugate.
Thus we divide our final count by 2, arriving at the weight

\[\frac{1}{2}\left(\gamma + \prod_{i=1}^{k} \delta_i\right)\]
\noindent given in Theorem~\ref{OhtsukiSmoothings}.

\end{proof}

\section{A family of pretzel knots}
To complete the paper, we discuss the possible extension of these ideas to computations for Montesinos knots. We believe that this exhaustive treatment for 2-bridge knots can be extended to Montesinos knots broadly. Our development of trees arises directly from Ohtsuki's lemma, given here as Lemma~\ref{minimalsubtrees}, and in this broad sense we offer nothing new. However our development of subtrees associated to twists makes the project of assembling trees to compute weights more realistic from a practical standpoint. Essentially, we have reduced the tree assembly problem from a sort of jigsaw puzzle with a very large number of puzzle pieces to a much smaller puzzle with only a handful of pieces. 

Note that the boundary slopes of incompressible surfaces for all Montesinos knots were identified in \cite{HO}. In fact Nathan Dunfield has created a program for generating these \cite{D}. Thus the computational tool needed for the computation of the CGLS-seminorms and related invariants is really the weights, which are read easily from the trees.

To illustrate, we have found what we conjecture are the nine basic trees corresponding to the incompressible surfaces in the complement of the $(p,q,r)$ pretzel knot, where $p$, $q$, and $r$ are all odd and greater than 1. Below we provide our conjectural solution. We do not offer a theoretical justification here, as it is well beyond the scope of this paper. In particular, one would need to confirm that all angles can be realized and that conjugate and reducible actions are properly accounted for. We provide these trees as motivation for the reader to pursue these investigations more fully. 

The reader should be aware that for these knots there are incompressible surfaces of two types, Type III and Type II, per \cite{HO}. The Type III surfaces below may be constructed as state surfaces, as in the current paper. The Type II surfaces are not state surfaces.

For clarity, we have drawn trees for the case $(p,q,r) = (3,5,7)$, and we trust the general case is easily understood from this example. The weights in the conjecture are given for arbitrary odd positive values of $p$, $q$ and $r$.

\pagebreak

\begin{conjecture}
The weights for the nine incompressible surfaces for the $(p,q,r)$-pretzel knot, where $p$, $q$, and $r$ are odd and greater than 1, may be computed using trees completely analogous to those shown in Figures~\ref{fig:pretzelstate} and ~\ref{fig:pretzelnonstate} for the case $(p,q,r) = (3,5,7)$. The surfaces have weights as follows, following the order shown in Figures~\ref{fig:pretzelstate} and ~\ref{fig:pretzelnonstate}:
\begin{list}{--}{}  
    \item[Type III surface, 
    Weight $1/2(p-1)(q-1)(r-1)$] 
      \item[Type III surface, 
      Weight $1/2(r-1)$ ] 
     \item[Type III surface, 
     Weight $1/2(q-1)$ ] 
     \item[Type III surface, 
     Weight $1/2(p-1)$ ] 
      \item[Type III surface, 
      Weight $1$.] 
      \item[Type II surface, 
      Weight $1/2(pqr - 1)$] 
       \item[Type II surface, 
      Weight $1/2(qr-1)$ ] 
        \item[Type II surface, 
        Weight $1/2(pr-1)$ ] 
      \item[Type II surface, 
       Weight $1/2(pq-1)$]     
\end{list}
\end{conjecture}

\begin{figure}

\centering

\includegraphics[width=13cm, height=20cm]{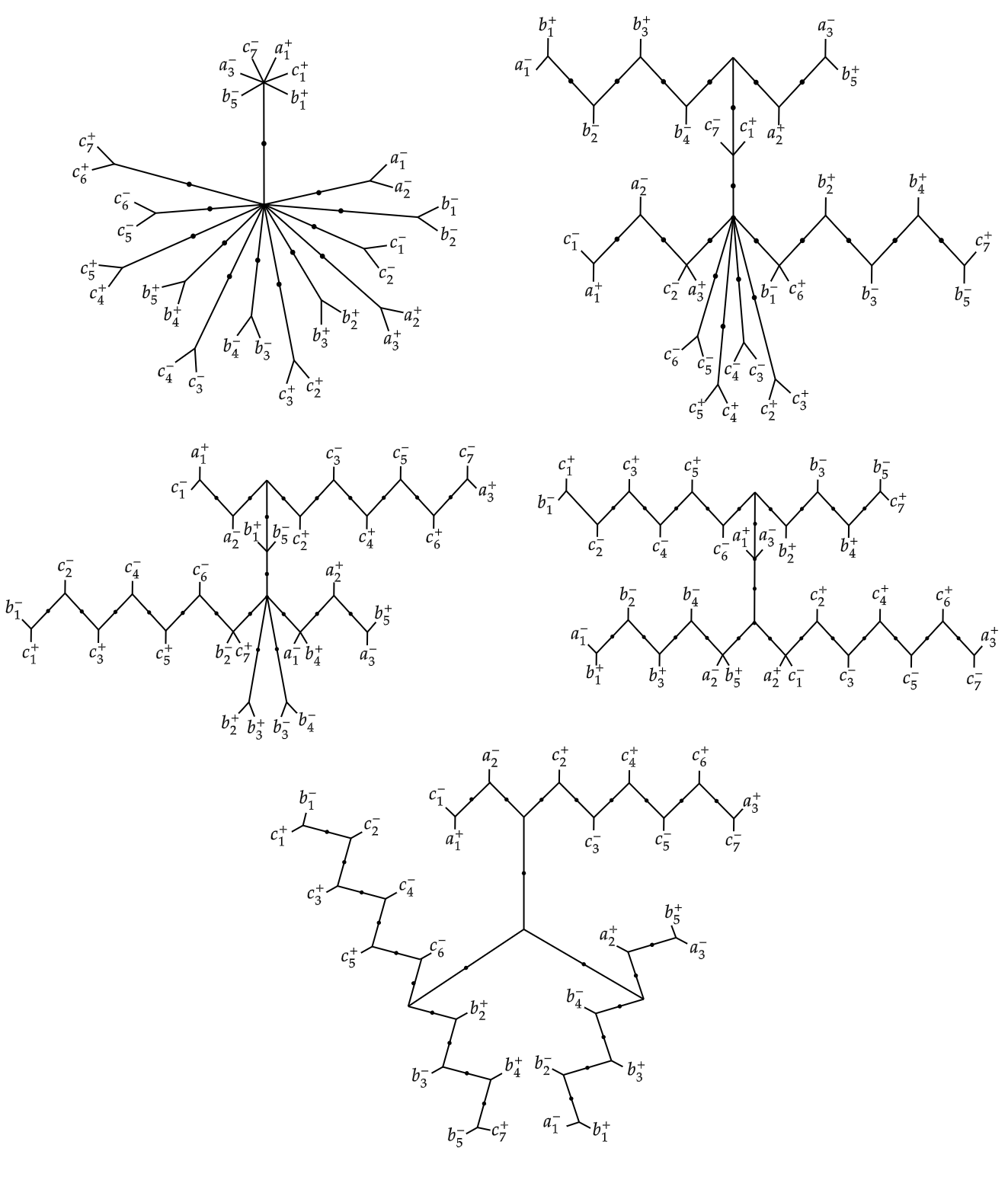}
\caption{Conjectural trees for Type III (state) surfaces for the (3,5,7) pretzel knot, slopes 0, -16, -20, -24, -30}
\label{fig:pretzelstate}
    
\end{figure}

\begin{figure}
    \centering

\includegraphics[width=13cm, height=20cm]{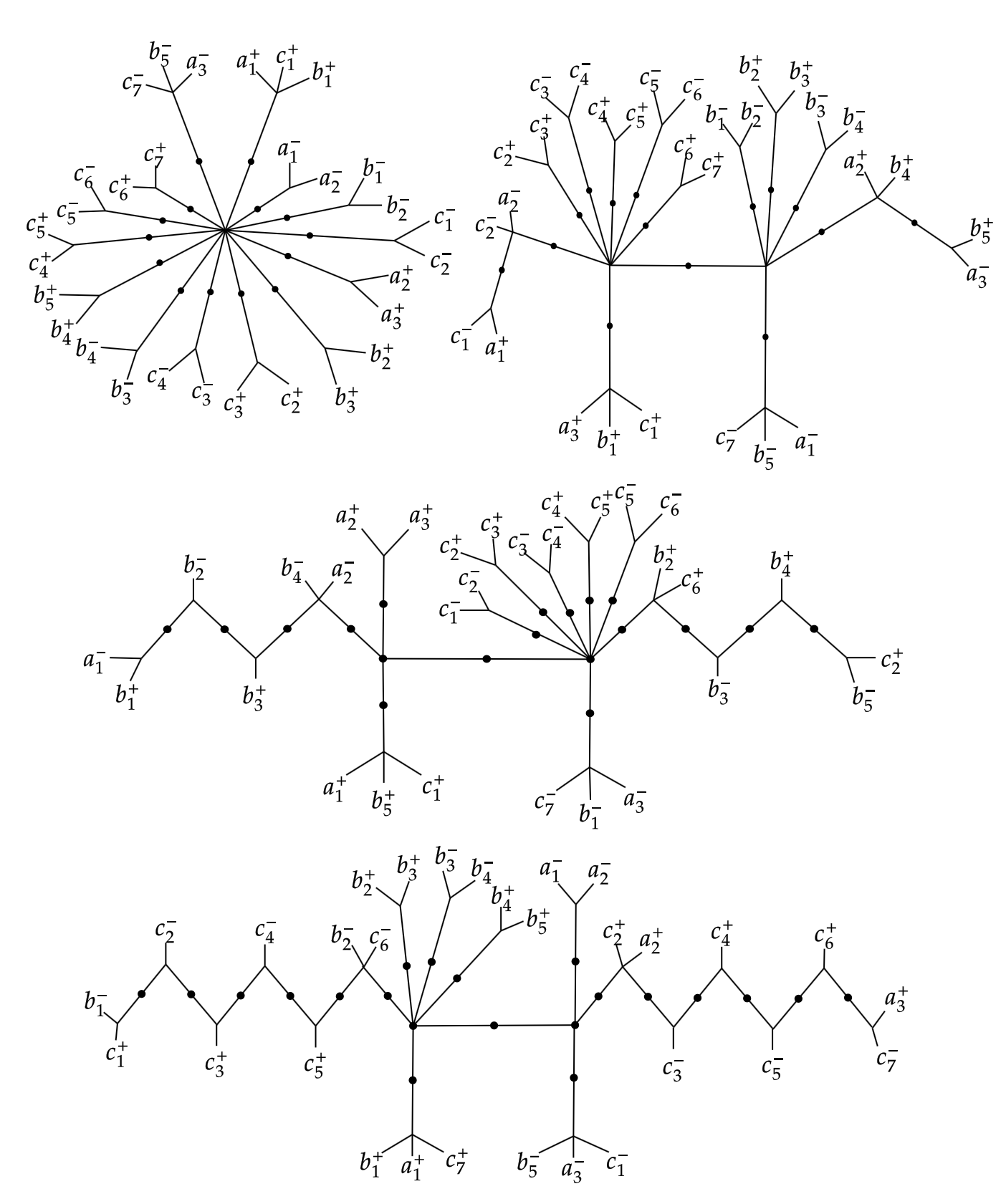}
\caption{Conjectural trees for Type II (non-state) surfaces for the (3,5,7) pretzel knot, slopes 0, -4, -8, -12}
\label{fig:pretzelnonstate}
\end{figure}

\bigskip
\noindent
{\it Acknowledgements.} We extend our thanks to Andrew Clifford for his conversations throughout this project and on the writing of this paper as well as for his careful reading of the undergraduate thesis of the third author.

 \pagebreak

\end{document}